\documentclass[leqno]{article}
\usepackage[papersize={176.76mm,250mm},margin=0.75in ]{geometry}

\usepackage{etex}
\usepackage{amsmath, bm,amssymb,amsfonts,dsfont,amsthm}
\usepackage{graphicx}
\usepackage[usenames,dvipsnames,svgnames,table]{xcolor}
\usepackage{multicol}
\usepackage{mathrsfs}
\usepackage{mathtools}
\usepackage{amsthm}
\usepackage{enumitem}
\usepackage{lineno}
\definecolor{antiquefuchsia}{rgb}{0.57, 0.36, 0.51}
\definecolor{auburn}{rgb}{0.43, 0.21, 0.1}
\definecolor{darkcerulean}{rgb}{0.03, 0.27, 0.49}
\definecolor{denim}{rgb}{0.08, 0.38, 0.74}
\definecolor{black}{rgb}{0.0, 0.0, 0.0}
\definecolor{sacramentostategreen}{rgb}{0.0, 0.34, 0.25}
\definecolor{phthaloblue}{rgb}{0.0, 0.06, 0.54}
\usepackage[colorlinks=true,linkcolor=black, urlcolor=auburn, filecolor=blue, citecolor=sacramentostategreen,backref=page]{hyperref}

\usepackage{etoolbox}
\makeatletter
\patchcmd{\BR@backref}{\newblock}{\newblock(Cited on page~}{}{}
\patchcmd{\BR@backref}{\par}{)\par}{}{}
\makeatother

\usepackage{relsize}

\usepackage[all]{xy}
\usepackage{tikz-cd}
\usepackage{array}
\usepackage{tensor}
\usepackage[cal=mathpi,calscaled=.94, bb=ams,frakscaled=.97,scr=rsfso]{mathalfa} 
\usepackage[utf8]{inputenc}

\usepackage{pdflscape}
\usepackage{float}

\usepackage{scalerel,stackengine}

\usetikzlibrary{matrix}
\usetikzlibrary{arrows}

\newcommand{\ds}[1]{\ensuremath{\mathds{#1}}}

\newcommand{\curly}[1]{\ensuremath{\mathscr{#1}}}

\newcommand{\dv}{\ensuremath{\mathrm{Div}}}

\newcommand{\ord}{\ensuremath{\mathrm{ord}}}

\DeclareMathOperator{\kr}{Ker}

\newcommand{\tio}[1]{\ensuremath{{#1}^\times}}

\newcommand{\lr}[1]{\left\langle{#1}\right\rangle}

\newcommand{\id}[1]{\ensuremath{{\mathfrak{#1}}}}

\newcommand{\abs}[1]{\ensuremath{{\left\vert{#1}\right\vert}}}

\newcommand{\bs}{\ensuremath{\backslash}}
\newcommand{\xra}[1]{\xrightarrow{#1}}
\newcommand{\ra}{\rightarrow}
\newcommand{\til}[1]{\widetilde{#1}}

\newcommand{\alg}[1]{{{#1}^{\mathrm{alg}}}}

\usepackage{ccfonts,eucal}
\usepackage[euler-digits,euler-hat-accent]{eulervm}
\usepackage{sseq}

\newtheorem{theorem}{Theorem}[section]

\newtheorem{proposition}[theorem]{Proposition}
\newtheorem{corollary}[theorem]{Corollary}
\newtheorem{mdef}[theorem]{Definition}
\let\olddefinition\mdef
\renewcommand{\mdef}{\olddefinition\normalfont}

\newtheorem{exam}[theorem]{Example}
\let\oldexample\exam
\renewcommand{\exam}{\oldexample\normalfont}

\newtheorem{rem}{Remark}
\let\oldremark\rem
\renewcommand{\rem}{\oldremark\normalfont}

\stackMath
\newcommand\rwhat[1]{%
\savestack{\tmpbox}{\stretchto{%
  \scaleto{%
    \scalerel*[\widthof{\ensuremath{#1}}]{\kern-.6pt\bigwedge\kern-.6pt}%
    {\rule[-\textheight/2]{1ex}{\textheight}}
  }{\textheight}%
}{0.5ex}}%
\stackon[1pt]{#1}{\tmpbox}%
}
\usepackage{changepage}
\usepackage{fancyhdr} 
\newcommand{\myreferences}{C:/Users/Harpreet/Documents/Biblio/bibThesis}

\begin{document}

\title{Perfectoid Tate Curves}
\author{Harpreet Singh Bedi~~~{bedi@gwu.edu}}

\maketitle

\begin{abstract}
Perfectoid versions of Abel Jacobi and Reimann Roch Theorem are proved, and perfectoid Elliptic Curve is constructed. A Perfectoid Tate Curve is defined and its cohomology computed via a \v{C}ech complex. Furthermore, perfectoid Theta function and Weierstra{\ss} series are also defined and suitably interpreted.
\end{abstract}

\tableofcontents

\section{Introduction} The story begins by considering power series of the form given in \eqref{eq:perf1} with coefficients in the perfectoid field $K$ denoted by $K\lr{X}_\infty$. Perfectoid fields are defined in \cite{scholze_1} and most foundational details can found in \cite{Kiran_AWS}.

\begin{equation}\label{eq:perf1}
\sum_{n\geq 0}a_nX^n,~n\in\ds{Z}[1/p]\text{ and } \abs{a_n}\ra 0 \text{ as } n\ra\infty
\end{equation}

The above comes equipped with a Gauss norm, the valuation ring $R=\{a\in K:\abs{a}\leq 1\}$, a  maximal ideal $\id{m}=\{a\in K:\abs{a}< 1\}$ and the residue field $k=R/\id{m}$. Let $R$ denote the restricted ring with $\abs{f}\leq 1$ (Gauss Norm). The elements of the series can be ordered by observing the countability of rationals.

The reduction map takes power series and converts them into polynomials
\begin{equation}
\begin{aligned}
\pi:R&\ra k\\
\pi: R\lr{X}_\infty&\ra k[X,X^{1/p},\ldots,X^{1/p^i},\ldots ]\\
f&\mapsto \til{f}
\end{aligned}
\end{equation}
In particular $g\in K\lr{X}_\infty$ is a unit iff its reduction is a unit $\til{g}\in\tio{k}$.
All the standard properties of Tate Algebras as described in \cite[pp 15]{bosch2014lectures} hold here, and have been proved in \cite{Bedi2018}. The above helps us define order of a power series (analogue of degree of polynomials).

\begin{mdef}
A power series $g\in K\lr{X}_\infty$ with $\abs{g}=1$ is distinguished of order $s$ iff its reduction is of the form
\begin{equation}
\til{g}=a_0+\ldots+a_iX^{j/p^i}+\ldots+a_sX^s, a_i\in\tio{K}, s\in\ds{Z}[1/p]
\end{equation}
\end{mdef}

\begin{theorem}[Weirstra{\ss} preparation theorem]
Let $g\in K\lr{X}_\infty$ of order $s$, then there is a unique monic polynomial $h$ of degree $s$ such that $g=uh$ where $u$ is unit in $K\lr{X}_\infty$.
\end{theorem}

Notice that an element $g\in K\lr{X}_\infty$ has only finitely many zeros, since $h$ in the theorem above can be made into a polynomial with integer degrees by a change of variable. For example $X^5+X^{1/p^2}$ has degree $5p^2$. Thus, for any polynomial look for the minimum power which would be of the form $a/p^i$ ($i=0$ for integer) and then change the variable $X^{1/p^i}\mapsto X.$ The result is the key to doing algebraic geometry on perfectoid spaces. Start with series, do the computations and reduce to polynomial forms whenever one wants to talk about zeros or poles.

\section{Perfectoid Abel Jacobi Theorem}

\begin{mdef}
Given a point $\alpha\in K$ we define $\ord_\alpha g$ as the highest power of $\varphi_\alpha(X)$ to divide $h$, where $\varphi_\alpha(X)$ is the irreducible polynomial of $\alpha$ over $K$.
\end{mdef}

Put $\ord_\alpha(0)=+\infty$ and obtain an additive valuation $g\mapsto \ord_\alpha(g)\in\ds{Z}[1/p]$, this can be extended to rational functions of the form $f=g_1/g_2$ with $g_i\in K\lr{X}_\infty$. The rational functions would also be called \emph{meromorphic functions}.

 The numbers $\ord_\alpha(f)$ satisfy the following:
\begin{enumerate}
\item {[Finiteness Condition]} There are only finitely many $\alpha$ with $\ord_\alpha\neq 0$ in every region $0<r\leq |\alpha|\leq r'$.
\item {[Rationality Condition]} If $\alpha$ and $\beta$ are conjugate over $K$, $\ord_\alpha=\ord_\beta$ (share the minimal polynomial).
\end{enumerate}

The collection $\{\ord_\alpha(f)\}$ is called the divisor of $f$. The collection $\{ m_\alpha\},\alpha\in \alg{K}$ with $m_\alpha\in\ds{Z}[1/p]$ satisfying the above conditions is called a divisor with component wise addition and thus forms an additive group. 

Given a divisor of the form $\sum_im_{\alpha_i}[\alpha_i]$, with finitely many $\alpha_i$, the corresponding function $f$ with $\dv(f)=\sum_im_{\alpha_i}[\alpha_i]$ is given as 
\begin{equation}
\prod_{\abs{\alpha_i}\leq 1}\left(1-\frac{\alpha_i}{X}\right)^{e_{\alpha_i}m_{\alpha_i}}\prod_{\abs{\alpha_i}> 1}\left(1-\frac{X}{\alpha_i}\right)^{e_{\alpha_i}m_{\alpha_i}}
\end{equation}
 where $e_\alpha$ degree of separability of $\alpha$ over $K$, we can combine the conjugates and get 
 \begin{equation}
 \left(\frac{\varphi_\alpha(X)}{X^{n_\alpha}} \right)^{m_\alpha} \text{ if }\abs{\alpha}\leq 1\text{ and }\left(\frac{\varphi_\alpha(X)}{a_0} \right)^{m_\alpha} \text{ if }\abs{\alpha}> 1
 \end{equation}

 where $n_\alpha=[K(\alpha):K]$ and $a_0=N_{K(\alpha)/K}(\alpha)$ as in \cite[pp 11-12]{roquette1970analytic}. Closely following \cite[chapter 1]{roquette1970analytic} one could adapt to the case of $q\in\ds{Z}[1/p]$. 

A meromorphic function on $\alg{K}$ has period $q$ if it satisfies the functional equation $f(q^{-1}X)=f(X)$, these functions form a subfield denoted as $F_K(q)$ and called elliptic function field over $K$. If a function is $q$ periodic then $\dv( f)$ is $q$ periodic too. The non-zero meromorphic functions $f$ can be determined by $\dv(f)$ upto a function of the form $cX^d$ with $c\in\tio{K}$ and $d\in\ds{Z}[1/p]$, this helps define theta functions
\begin{equation}
f\left(\frac{X}{q}\right)=\frac{(-X)^df(X)}{a}\text{ where }a\in\tio{K}, d\in\ds{Z}[1/p]
\end{equation}
where $d$ is called degree of $f$ and $a$ is called multiplicator of $f$. If another function $f$ satisfies $\dv g=\dv f$ it becomes a theta function with same degree as $f$
\begin{equation}\label{Rq3}
g\left(\frac{X}{q}\right)=\frac{(-X)^dg(X)}{aq^k},\qquad k\in\ds{Z}[1/p]
\end{equation}
and its multiplicator differs by a power of $q$. Thus, $d$ is uniquely determined by the divisor (say $\id{d}=\dv f=\dv g$). The degree is uniquely determined as
\begin{equation}
\deg_q\id{d}=d.
\end{equation}
The multiplicator $a$ is uniquely determined upto a power of $q$, we denote its residue class in
$\tio{K}/q^k, k\in\ds{Z}[1/p]$ as $\Phi_q(\id{d})$ (\emph{Jacobi image}) , and write
\begin{equation}		
a\equiv \Phi_q(\id{d}){\mod^\times } q^{p^{-\infty}}
\end{equation}
If $f(X/q)=f(X)$, then \eqref{Rq3} gives
\begin{equation}\label{Roq5}
\deg_q(\id{d})=0\text{ and }\Phi_q(\id{d})\equiv 1 
\end{equation}

Conversely, if the above is satisfied then for any rational function with $\dv f=\id{d}$ we have $d=0,a=q^{-k}$ for some $k\in\ds{Z}[1/p]$. Then $g(X)$ defined below is $q$ periodic and is uniquely determined upto a factor in $\tio{K}$.
\begin{equation}
g(X)=cX^kf(X), \qquad c\in\tio{K}
\end{equation}
Thus, the conditions in \eqref{Roq5} are necessary and sufficient for $\id{d}$ to be a divisor of a $q$ periodic function. 
The formulas in \eqref{Roq12} explicitly give Jacobi image (of a $q$ periodic divisor $\id{d}$) and the degree. We have the Perfectoid-Abel-Jacobi Theorem {\cite[pp 15]{roquette1970analytic}}:
\begin{proposition}\label{prop1}[Perfectoid-Abel-Jacobi Theorem]
A $q$ periodic function $f$ has a $q$ periodic divisor $\id{d}$ iff it satisfies
\begin{equation}
\deg_q(\id{d})=0,\qquad \Phi_q(\id{d})\equiv 1\mod^{\times}q^{p^{-\infty}}
\end{equation}
with degree and Jacobi image $\Phi_q$ are given in \eqref{Roq12}. Furthermore, $f$ is uniquely determined by $\id{d}$ upto a factor in $\tio{K}$.
\end{proposition}
The fundamental theta function for $\id{d}$ gives an explicit formula for computing the degree and Jacobi image.
\begin{equation}\label{funtheta}
\Theta(X)=\prod_{n\geq 0}\left(1-\frac{q^n}{X}\right)\prod_{n<0}\left(1-{q^{-n}}{X}\right),\qquad n\in\ds{Z}[1/p]
\end{equation}
The above function is a  satisfies the functional equation $\Theta(X/q)=-X\Theta(X)$ which follows from the observation
\begin{equation}
\frac{\Theta(X)}{\Theta(X/q)}=\underset{n\geq 0}{\underbrace{\left(1-\frac{1}{X}\right)}}\cdot
\overset{n< 0}{\overbrace{\left(\frac{1}{1-X}\right)}}=-\frac{1}{X}\end{equation}

Rewriting the functional equation explicitly \eqref{thetaFunc} observe that degree of $\Theta$ is one and its multiplicator is also one.
\begin{equation}\label{thetaFunc}
\Theta(X/q)=-X\Theta(X)
\end{equation}
For every $\alpha\in\tio{\alg{K}}$ define
\begin{equation}
\Theta_\alpha(X)=\Theta(\alpha^{-1}X)
\end{equation}
which is a $q$ periodic function with multiplicity one its functional equation is
\begin{equation}
\Theta_\alpha(q^{-1}X)=\alpha^{-1}(-X)\Theta_\alpha(X)
\end{equation}
which gives degree one and multiplicator $\alpha$. Set
\begin{equation}
\Theta_{\id{d}}=\prod_{|q|<|\alpha|\leq 1}\Theta_\alpha^{e_\alpha m_{\alpha}}
\end{equation}
where $e_\alpha$ is the degree of inseparability of $\alpha$ over $K$ and $\id{d}=\{m_\alpha\}$ is a $q$ periodic divisor and $\Theta_\id{d}$ satisfies the functional equation
\begin{equation}
\begin{aligned}
\Theta_\id{d}(X/q)&=(-X)^d\Theta_{\id{d}}(X)\\
\text{where}\qquad\qquad\qquad &\\
d=\sum_{|q|<|\alpha|\leq 1}{e_\alpha m_\alpha}&=\deg_q\id{d}\\
a=\prod_{|q|<|\alpha|\leq 1}\alpha^{e_\alpha m_\alpha}&\cong \Phi_q(\id{d}){\mod^\times }q^{p^{-\infty}}
\end{aligned}
\end{equation}

The rationality conditions ensure that we can write the above in the form below with prime denoting the conjugacy classes (with conjugate elements having the same absolute value).
\begin{equation}
\begin{aligned}\label{Roq12}
\deg_q\id{m}&={\sum'}_{|q|<|\alpha|\leq 1}[K(\alpha):K]m_\alpha\\
\Phi_q(\id{d})&\cong \prod'_{|q|<|\alpha|\leq 1} N_{K(\alpha)|K}(\alpha)^{m_\alpha}\mod^\times q^{p^{-\infty}}
\end{aligned}
\end{equation}

The perfectoid version of Corollary at \cite[pp 15]{roquette1970analytic}

\begin{corollary}\label{coro1}
For every $\alpha\in\tio{\alg{K}}$ there is a $q$ periodic function $f$ with $\ord_\alpha(f)=1$. If $\beta$ is not $K$ conjugate of $\alpha\mod^\times q^{p^{-\infty}}$, then we can choose $f$ such that $\ord_\beta(f)=0$.
\end{corollary}
\begin{proof}
Following \eqref{Roq12} we construct divisor $\id{p}_\alpha$ for elements conjugate to $\alpha$
\begin{equation}
\deg_q\id{p}_\alpha=[K(\alpha):K],\qquad \Phi_q(\id{p}_\alpha)\equiv N_{K(\alpha)/K}(\alpha)\mod^\times q^{p^{-\infty}}
\end{equation}
with $\ord_\alpha(f)=1$ (for all $K$-conjugates of $\alpha$) and $\ord_\beta(f)=0$ by setting $m_\alpha=1$ (thus, there are no fractional powers that need to be considered in the proof).

There are two cases to consider either $\alpha\in\tio{K}$ or $\alpha\notin \tio{K}$, that is $\alpha\in\tio{\alg{K}}\bs\tio{K}$.\\
Case 1: For $\alpha\in \tio{K}$ choose elements $u,v\in\tio{K}$ such that $\alpha,uv,\alpha u, v$ are all different $\mod^\times q^{p^{-\infty}}$. The divisor $\id{d}$
\begin{equation}
\begin{aligned}
\id{d}&=\id{p}_\alpha+\id{p}_{uv}-\id{p}_{\alpha u}-\id{p}_v,\\
\deg\id{d}&=0\text{ and }\Phi_q(\id{d})\equiv\frac{\alpha uv}{\alpha uv}\equiv 1\mod^\times q^{p^{-\infty}}
\end{aligned}
\end{equation}
Thus, $\id{d}=\dv(f)$ a $q$ periodic function with $\ord_\alpha(f)=1$ and $u,v$ can be so chosen that $uv,\alpha u,v$ are all different from $\beta\mod^\times q^{p^{-\infty}}$ giving $\ord_\beta(f)=0$.\\
Case 2: Let $\alpha\notin \tio{K}$ and $d=[K(\alpha):K],a=N_{K(\alpha)|K}(\alpha), u\in\tio{K}, v=u^{1-d}a$, then the divisor
\begin{equation}
\begin{aligned}
\id{d}&=\id{p}_\alpha-(d-1)\id{p}_u-\id{p}_v\\
\deg{\id{d}}&=0\text{ and }\Phi_q\equiv \frac{a}{u^{d-1}v}\equiv 1\mod^\times q^{p^{-\infty}}.
\end{aligned}
\end{equation}
Thus $\dv f=\id{d}$ with a $f$ a $q$ periodic function and by construction $\id{d}$ has multiplicity one at $\alpha$, and an appropriate choice of $u,v\not\equiv\beta\mod^{\times}q^{p^{-\infty}}$ gives $\ord_\beta(f)=0$.
\end{proof}

\begin{mdef} The vector space of a $q$ periodic divisor is denoted as
\begin{equation}
\begin{aligned}
L_K(q|\id{d})&:=\{\text{$q$ periodic functions $f$ such that }\dv(f)\geq -\id{d}\}\\
\ell_k(q|\id{d})&:=\dim L_K(q|\id{d})
\end{aligned}
\end{equation}

\end{mdef}

\begin{rem}\label{rem1}
Proposition \ref{prop1} helps give the dimension
\begin{equation}
\ell_k(q|\id{d})
\begin{cases}
0 \text{ if }\deg_q{\id{d}}<0\\ 
1 \text{ if }\deg_q{\id{d}}=0\\ 
\end{cases}
\end{equation}

Note that in the perfectoid world $\deg\in\ds{Z}[1/p]$.

\end{rem}

In order to prove the Riemann-Roch Theorem in the perfectoid, the corollary \ref{coro1} has to be recast for a perfectoid power $1/p^i$ ($i,p$ fixed) in place of $1$.

\begin{corollary}\label{coro2}
For every $\alpha\in\tio{\alg{K}}$ and chosen $i,p$ there is a $q$ periodic function $f$ with $\ord_\alpha(f)=1/p^i$. If $\beta$ is not $K$ conjugate of $\alpha\mod^\times q^{p^{-\infty}}$, then we can choose $f$ such that $\ord_\beta(f)=0$.
\end{corollary}

\begin{proof}
Following \ref{coro1} we construct divisor $\id{p}_\alpha$ for elements conjugate to $\alpha$
\begin{equation}
\deg_q\id{p}_\alpha=[K(\alpha):K]\cdot\frac{1}{p^i},\qquad \Phi_q(\id{p}_\alpha)\equiv N_{K(\alpha)/K}(\alpha)^{{1}/{p^i}}\mod^\times q^{p^{-\infty}}
\end{equation}
with $\ord_\alpha(f)={1}/{p^i}$ (for all $K$-conjugates of $\alpha$) and $\ord_\beta(f)=0$ by setting $m_\alpha={1}/{p^i}$ (thus, there are fractional powers that need to be considered in the proof). The divisor $\id{p}_\alpha$ has multiplicity $1/p^i$ at $\alpha$.

There are two cases to consider either $\alpha\in\tio{K}$ or $\alpha\notin \tio{K}$, that is $\alpha\in\tio{\alg{K}}\bs\tio{K}$.\\
Case 1: For $\alpha\in \tio{K}$ choose elements $u,v\in\tio{K}$ such that $\alpha,uv,\alpha u, v$ are all different $\mod^\times q^{p^{-\infty}}$. The divisor $\id{d}$
\begin{equation}
\begin{aligned}
\id{d}&=\id{p}_\alpha+\id{p}_{uv}-\id{p}_{\alpha u}-\id{p}_v,\\
\deg\id{d}&=0\text{ and }\Phi_q(\id{d})\equiv\frac{\alpha uv}{\alpha uv}\equiv 1\mod^\times q^{p^{-\infty}}
\end{aligned}
\end{equation}
Thus, $\id{d}=\dv(f)$ a $q$ periodic function with $\ord_\alpha(f)=1$ and $u,v$ can be so chosen that $uv,\alpha u,v$ are all different from $\beta\mod^\times q^{p^{-\infty}}$ giving $\ord_\beta(f)=0$.\\
Case 2: Let $\alpha\notin \tio{K}$ and $d=[K(\alpha):K],a=N_{K(\alpha)|K}(\alpha)^{1/p^i}, u\in\tio{K}, v=u^{1-d}a$, then the divisor
\begin{equation}
\begin{aligned}
\id{d}&=\id{p}_\alpha-(d-1)\id{p}_u-\id{p}_v\\
\deg{\id{d}}&=0\text{ and }\Phi_q\equiv \frac{a}{u^{d-1}v}\equiv 1\mod^\times q^{p^{-\infty}}.
\end{aligned}
\end{equation}
Thus $\dv f=\id{d}$ with a $f$ a $q$ periodic function and by construction $\id{d}$ has multiplicity $1/p^i$ at $\alpha$, and an appropriate choice of $u,v\not\equiv\beta\mod^{\times}q^{p^{-\infty}}$ gives $\ord_\beta(f)=0$.
\end{proof}
\subsection{Perfectoid Riemann-Roch}
Let a divisor be $1[x_1]+2[x_2]+3/p^i[x_3]$, this is rewritten as $1p^i/p^i[x_1]+2p^i/p^i[x_2]+3/p^i[x_3]$ and called a divisor with denominator $1/p^i$. The Riemann-Roch theorem gives the dimension $\ell_K(q|\id{d})$ as the degree ($\times p^i$) of the divisor. Setting $i=0$ gives the standard case as in \cite[Proposition 2, pp 16]{roquette1970analytic}.
\begin{theorem}\label{RR}
If $\id{d}$ is a $q$ periodic divisor with denominator $1/p^i$ (where $i,p$ are fixed) and $\deg_q(\id{d})>0$ then
\begin{equation}
\ell_K(q|\id{d})=\deg_q(\id{d})\cdot p^i
\end{equation}
\end{theorem}

\begin{proof}
For $\deg_q(\id{d})=0$ the result holds from Proposition \ref{prop1}, thus we may assume $d=\deg_q(\id{d})>0$ and use induction for the numerator of $d$. Start by choosing element $a\in\tio{K},a\equiv\Phi_q(\id{d})\mod^\times q^{p^{-\infty}}$, now choose $b\not\equiv 1, a\mod^\times q^{p^{-\infty}} $ such that is $\id{d}$ has multiplicity zero at $b$. The divisor $\id{p}_b$ as constructed in corollary \ref{coro2} has multiplicity $1/p^i$. Notice the divisor
\begin{equation}
\id{d}'=\id{d}-\id{p}_b\text{ where }\deg\id{d}'=(d-1)/p^i,\qquad \Phi_q(\id{d}')\equiv\frac{a}{b}\not\equiv 1 \mod^\times q^{-p^{-\infty}}
\end{equation}
By induction $\ell_K(q|\id{d}')=d-1$. Consider the map $f\ra f(b)$, the kernel of this map is the space $L_K(q|\id{d}')$ which has value zero at $b$. Since, we need one $d$ as dimension we need to show that there is an $f\in L_K(q|\id{d})$ and $f(b)\neq 0$. Let
\begin{equation}
\dv f=\id{z}-\id{d},\text{ hence }\id{z}\geq 0
\end{equation}
and $\id{z}$ has multiplicity zero at $b$. \ref{prop1} requires to show
\begin{equation}
\deg_q(\id{z})=d/p^i, \Phi_q(\id{z})\equiv a\mod^\times q^{p^{-\infty}}
\end{equation}
which is given by the divisor below
\begin{equation}
\id{z}=\id{p}_a+(d-1)\id{p}_1
\end{equation}
where $\id{p}_a$ and $\id{p}_1$ are as given in corollary \ref{coro2}(to transfer $1/p^i$) and $b\neq 1,a\mod^\times q^{-p^{-\infty}}$ as assumed above.
\end{proof}

\section{Perfectoid Tate Curve}

Define perfectoid Tate curve  as $\mathcal{T}:=\ds{G}_{m,K}/\lr{q}$ with $0<\abs{q}<1$, and $\lr{q}$ is the subgroup of $\tio{K}$ generated by $q$. This is precisely the same as the definition in \cite[pp 121]{fresnel2012rigid}. It is possible to define the above as $\mathcal{T}:=\ds{G}_{m,K}/\lr{q}$ with fractional powers for $q$, that is modulo out with subgroup generated by $q^{\ds{Z}[1/p]}$, but this will give us a different model. In the standard model \cite[pp 220]{bosch2014lectures} the gluing is induced via multiplication by $q$, if we go by the fractional power case we will have to consider gluing induced by multiplication by the vector $(q,q^{1/p},q^{1/p^2},\ldots,q^{1/p^i},\ldots)$, and define admissible sets according to each $q^{1/p^i}$. In order to simplify the situation we want to keep the same admissible sets as the standard Tate Curve but put a different sheaf on it. 

Let $\ds{B}(r_1,r_2)$ denote an annulus with inner radius $r_1$ and outer radius $r_2$, that is $|r_1|\leq |r_2|$. Following definition 1.3.1 \cite[pp6]{lutkebohmert2016rigid} the ring corresponding $\ds{B}(r_1,r_2)$ is $K\lr{X/r_2,r_1/X}$. We can replace  $\ds{B}(r_1,r_2)$ with  $\ds{B}(r_1^{1/p},r_2^{1/p})$ with corresponding ring as $K\lr{(X/r_2)^{1/p},(r_1/X)^{1/p}}$, and use direct image sheaf to transfer the ring $K\lr{(X/r_2)^{1/p},(r_1/X)^{1/p}}$ to the disk $\ds{B}(r_1,r_2)$.
Notice the inverse system below which is analogous to one given at \cite[pp16]{lutkebohmert2016rigid}.
\begin{equation}
\cdots\ra\ds{B}(r_1^{1/p^2},r_2^{1/p^2})\xra{(\cdot)^p}\ds{B}(r_1^{1/p},r_2^{1/p})\xra{(\cdot)^p}\ds{B}(r_1,r_2)
\end{equation}
which gives us a direct system (all maps are inclusion)
\begin{equation}
\cdots\leftarrow K\lr{(X/r_2)^{1/p^2},(r_1/X)^{1/p^2}}\leftarrow K\lr{(X/r_2)^{1/p},(r_1/X)^{1/p}}\leftarrow K\lr{X/r_2,r_1/X}
\end{equation}

Thus, we can construct inverse and direct limit of the systems above and call them perfectoid versions denoting them by $\ds{B}_\infty$ and the corresponding ring as $K\lr{(X/r_2)^{1/p^\infty},(r_1/X)^{1/p^\infty}}$. 
\[K\lr{X/r_2,r_1/X}_\infty:=K\lr{(X/r_2)^{1/p^\infty},(r_1/X)^{1/p^\infty}}=\cup_{i\geq 0}K\lr{(X/r_2)^{1/p^i},(r_1/X)^{1/p^i}}\]

We will not worry much about limits, instead we directly associate the ring $K\lr{X/r_2,r_1/X}_\infty$ to the disk $\ds{B}(r_1,r_2)$.

We will follow chapter $5$ of \cite{fresnel2012rigid} and define the open sets as $U_0=\ds{B}(q,q^{-1}),U_1=\ds{B}(q^2,q),U_{0,1,+}=\ds{B}(q,q),U_{0,1,-}=\ds{B}(q^2,q^2),U_{0+}=\ds{B}(q^{-1},q^{-1}) $. 

The corresponding rings are given below (with coefficients tending to zero as $n\ra \infty$). Notice that we are closely following \cite[pp 122]{fresnel2012rigid} replacing $z$ with $X$ and $\pi$ with $q$, and explicitly writing the constants (and of course $n\in\ds{Z}[1/p]$ instead of usual $\ds{Z}$).

\begin{landscape}
{\begin{equation}\label{eq:Tate1}
\begin{aligned}
\curly{O}(U_0)&=\left\{K\lr{qX,q/X}_\infty\text{ or }\sum_{n> 0}a_n(q X)^n+b_0+\sum_{n>0}b_n\left(\frac{q}{X}\right)^n\text{ with }\lim a_n=0,\lim b_n=0 \right\}\\
\curly{O}(U_1)&=\left\{K\lr{X/q,q^2/X}_\infty\text{ or }\sum_{n> 0}c_n\left(\frac{X}{q}\right)^n+c_0+\sum_{n>0}d_n\left(\frac{q^2}{X}\right)^n \text{ with }\lim c_n=0,\lim d_n=0\right\}\\
\curly{O}(U_{0,1,+})&=\left\{K\lr{X/q,q/X}_\infty\text{ or }\sum_{n> 0}e'_n\left(\frac{X}{q}\right)^n+e_0+\sum_{n>  0}e''_n\left(\frac{q}{X}\right)^n\text{ or }\sum_{n\in\ds{Z}[1/p]}e_n\left(\frac{X}{q}\right)^n\text{ with }\lim e_n=0\right\}\\
\curly{O}(U_{0,1,-})&=\left\{K\lr{X/q^2,q^2/X}_\infty\text{ or }\sum_{n> 0}f'_n\left(\frac{X}{q^2}\right)^n+f_0+\sum_{n>  0}f''_n\left(\frac{q^2}{X}\right)^n\text{ or }\sum_{n\in\ds{Z}[1/p]}f_n\left(\frac{X}{q^2}\right)^n\text{ with }\lim f_n=0\right\}\\
\curly{O}(U_{0+})&=\left\{K\lr{qX,1/qX}_\infty\text{ or }\sum_{n> 0}g'_n\left(qX\right)^n+g_0+\sum_{n>  0}g''_n\left(\frac{1}{qX}\right)^n\text{ or }\sum_{n\in\ds{Z}[1/p]}g_n\left(qX\right)^n
\text{ with }\lim g_n=0\right\}\\
\curly{O}(U_{0,1})&=\curly{O}(U_{0,1,+})\oplus\curly{O}(U_{0,1,-})
\end{aligned}
\end{equation}
}
\end{landscape}
\subsection{Gluing the Sets }
The inner boundary of $U_0$ is $U_{0,1+}$ and outer boundary is $U_{0+}$, and the outer boundary of $U_1$ is $U_{0,1+}$ and inner boundary is $U_{0,1,-}$. We can identify $U_{0,1,-}=\ds{B}(q^2,q^2)$ to $U_{0+}=\ds{B}(q^{-1},q^{-1})$ by multiplying with $1/q^3$. In terms of ring map $\curly{O}(U_{0+})\ra\curly{O}(U_{0,1,-})$ the mapping is $X\mapsto X/q^3$( or $qX\mapsto X/q^2$). 
 
 \begin{equation}
 \ds{B}(q^2,q^2)=U_{0,1,-}\xra{1/q^3}U_{0+}=\ds{B}(q^{-1},q^{-1})
 \end{equation}

The above gluing is necessary for identification of rings in the Cech complex. 

The mapping from $\curly{O}(U_1)$ to $\curly{O}(U_{0,1,+})\oplus \curly{O}(U_{0,1,-})$ carries terms with coefficient $c_n$ to $e'_n$ and $d_n$ to $f''_n$, we rewrite this as $(c_n,d_n)\mapsto(e'_n,f''_n) $. Similarly, we have the mapping $\curly{O}(U_0)$ to $\curly{O}(U_{0,1,+})\oplus \curly{O}(U_{0+})=\curly{O}(U_{0,1,-})$ where $(a_n,b_n)\mapsto (g'_n,e''_n)\mapsto (f'_n,e''_n)$. 

\begin{figure}[H]
\centering
 \begin{tikzpicture}\label{check8}
 []
        \matrix (m) [
            matrix of math nodes,
            row sep=0.5em,
            column sep=7.5em,
                   ]
{   |[name=aa]|\curly{O}(U_0) &  |[name=ab]|\curly{O}(U_{0,1,+}) \\
 |[name=ka]| \oplus &|[name=kb]| \oplus \\
 |[name=qa]| \curly{O}(U_1) &  |[name=qb]|\curly{O}(U_{0,1,-}) \\
 |[name=wa]| \sum_{n> 0}a_n(q X)^n &  |[name=wb]|\sum_{n> 0}e'_n\left(\dfrac{X}{q}\right)^n \\
  |[name=ea]| +b_0  &  |[name=eb]|+e_0 \\
   |[name=ra]|+\sum_{n>0}b_n\left(\dfrac{q}{X}\right)^n &  |[name=rb]|+\sum_{n>  0}e''_n\left(\dfrac{q}{X}\right)^n \\
   \oplus & \oplus \\
  |[name=ta]| \sum_{n> 0}c_n\left(\dfrac{X}{q}\right)^n &  |[name=tb]|\sum_{n> 0}f'_n\left(\dfrac{X}{q^2}\right)^n \\
  |[name=ya]| +c_0 &  |[name=yb]|+f_0 \\
   |[name=ua]|+ \sum_{n>0}d_n\left(\dfrac{q^2}{X}\right)^n &  |[name=ub]|+\sum_{n>  0}f''_n\left(\dfrac{q^2}{X}\right)^n \\ 
        };
 \path[overlay,->, font=\scriptsize,>=latex]
        (wa) edge [out=355,in=175,looseness=1] (tb)
         (ra) edge (rb)
         (ua) edge (ub)

;

\path[overlay,->,color=gray, font=\scriptsize,>=latex]
         (ya) edge [out=355,in=195,looseness=1.5] (eb)
         (ea) edge (eb)
         ;
         
         \path[overlay,->,color=blue, font=\scriptsize,>=latex]
         (ta) edge [out=355,in=175,looseness=1] (wb)
         (ua) edge (ub)
         ;
\end{tikzpicture} 
\caption{Restriction Maps: Sets restricted to their boundary}\label{check8}  
  \end{figure}

The \v{C}ech complex is given as 
\begin{equation}
\begin{aligned}
\curly{O}(U_0)\oplus \curly{O}(U_1)&\xra{d} \curly{O}(U_{0,1,+})\oplus \curly{O}(U_{0,1,-}) \xra{d_1} 0\\
(a_n,b_n)\oplus (c_n,d_n)&\xra{d}(b_n-c_n,a_n-d_n)\\
~~&~~~~~(e''_n-e'_n,f'_n-f''_n)\\
(K,0)\oplus (K,0)&\xra{d}(0,0)
\end{aligned}
\end{equation}
Notice that for $b_0=c_0=$ constant, the $\kr d=0$. Hence we get the global sections $H^0(\curly{O}_\mathcal{T},\mathcal{T}_p)=K$. 
\begin{equation}
\begin{aligned}
\curly{O}(U_0)\oplus \curly{O}(U_1)\xra{d} \curly{O}(U_{0,1,+})\oplus \curly{O}(U_{0,1,-}) &\xra{d_1} 0\\
(e'_n+e''_n\oplus f'_n+f''_n)&\xra{d_1}(0,0)
\end{aligned}
\end{equation}
But, the above terms can be lifted to $(b_n,-c_n,a_n,-d_n)$ or $(a_n,b_n)\oplus(-c_n,-d_n)$. We now consider the constant terms $(e_0,f_0)$, where $e_0$ can be lifted to $(b_0-c_0)$ but $f_0$ cannot be lifted. Thus, we get $\dim_KH^1(\curly{O}_\mathcal{T},\mathcal{T}_p)=1$.

\begin{theorem}
The cohomology groups for perfectoid Tate Curve are given as
\begin{equation}
H^i(\curly{O}_{\mathcal{T}_p},\mathcal{T}_p)=
\begin{cases}
K & \text{~for~} i=0,1\\
0 & \text{~for~} i\geq 2
\end{cases}
\end{equation}
\end{theorem}

For some computations we want $d$ to be surjective, which can be accomplished by multiplying $\curly{O}(U_1)$ by $(X-1)$ which introduces the constant term $d_{-1}$ (in which we absorb $-c_0$). This makes the map $d$ onto in the following \v{C}ech complex.
\begin{equation}\label{surjComp}
\begin{aligned}
\curly{O}(U_0)\oplus (X-1)\curly{O}(U_1)\xra{d} \curly{O}(U_{0,1,+})\oplus \curly{O}(U_{0,1,-}) &\xra{d_1} 0\\
(K,0)\oplus (0,K)&\xra{d}(K,K)
\end{aligned}
\end{equation}
In place of $(X-1)$ we can also use a factor $X^n-1$ where $n\in\ds{Z}[1/p]$ and work with constant term $d_{-n}$.

\section{Divisors on $\mathcal{T}_p$} A divisor $D$ on $\mathcal{T}$ is a finite formal finite sum  of the form
\begin{equation}
D=\sum_{i=1}^sn_i[x_i]\text{ where }n_i\in\ds{Z}[1/p], x_i\in\mathcal{T}\text{ and }\deg D=\sum_in_i
\end{equation}
A positive divisor has all $n_i\geq 0$ and is denoted by $D\geq 0$. The sheaf of meromorphic functions($\curly{M}(U_i)$) are defined as the ring of quotients of $\curly{O}(U_i)$. A divisor of a non zero meromorphic function $f$ is defined as 
\begin{equation}
\dv(f)=\sum_{x\in\mathcal{T}}\ord_xf[x]
\end{equation}

Notice that a function $f$ will only have a finite number of zeros and thus the $\dv(f)$ will make sense and give a finite sum. Also, notice that $\curly{O}(U_0)$ is not a PID, but a Bezout domain. There exists a holomorphic function $h_0$ on $U_0$ such that the divisor of $h_0$ on $U_0$ is $D$. This is possible because we can first work for the Tate case and get a function $h_0$ as described on \cite[pp 124]{fresnel2012rigid} and then replace integer powers with desired fractional powers. We define the sheaf of divisors as 
\begin{equation}
\curly{L}(D)(U)=\{f\in\curly{M}(U)\text{ such that } \dv(f)\geq -D|_U\}
\end{equation} 

For a positive divisor $D$ we have a SES with $\curly{Q}$ a coherent sheaf with finite support (skyscraper sheaf).
\begin{equation}
0\ra\curly{O}_\mathcal{T}\ra\curly{L}(D)\ra\curly{Q}\ra 0
\end{equation}

We know that $H^i(\mathcal{T},\curly{Q})=0$ for $i\geq 1$ and $H^i(\mathcal{T},\curly{O}_{\mathcal{T}_p})=0$ for $i\geq 2$. We get that $H^i(\mathcal{T}, \curly{L}(D))=0$ for $i\geq 2$. Considering Euler Characters for the SES of sheaves we get
\begin{align}
\chi(\curly{L}(D))=\chi(\curly{O}_\mathcal{T})+\chi(\curly{Q})
\end{align}
Since, $\chi(\curly{O}_\mathcal{T})=0$ (the zero and one dim are both $1$ and cancel out) we get 
$\chi(\curly{L}(D))=\chi(\curly{Q})$.

We have proved the following

\begin{theorem}\label{thm5.1.2.1}
For any perfectoid divisor on $\mathcal{T}$ we have the following
\begin{enumerate}
\item $H^i(\mathcal{T}, \curly{L}(D))=0$ for $i\geq 2$
\item $\dim H^0(\mathcal{T}, \curly{L}(D))-\dim H^1(\mathcal{T}, \curly{L}(D))=\dim H^0(\mathcal{T},\curly{Q})$
\end{enumerate}
\end{theorem}

The Perfectoid Riemann Roch theorem \ref{RR} gives $\dim H^0(\mathcal{T},\curly{Q})=\deg D\cdot p^i$ for $i,p$ fixed in the divisor $D$. This should be thought of as a simple change of variable $X^{1/p^i}\mapsto X$. 
\begin{theorem}\label{thm5.1.2.2}
For any perfectoid divisor on $\mathcal{T}$, $H^1(\mathcal{T}, \curly{L}(D))=0$ for $\deg D>0$. 
\end{theorem}
The two theorems give $\dim H^0(\mathcal{T}, \curly{L}(D))=\deg D\cdot p^i$, which can be used to construct perfectoid elliptic curves. The proof closely follows \cite[pp 125]{fresnel2012rigid}.
\begin{proof}
From theorem \ref{thm5.1.2.1}, there is a non zero meromorphic function $f$ such that $\dv f\geq -D$ (for $D>0$), and $f$ provides isomorphism between line bundles $\curly{L}(D)$ and $\curly{L}(\tilde{D})$ where $\tilde{D}=D+\dv f$. Hence, it suffices to work with any $\deg D>0$. Let $D\geq 1/p^i[t]$ for any $t\in \mathcal{T}$ and consider the exact sequence
\begin{equation}
0\ra\curly{L}(1/p^i[t])\ra\curly{L}(D)\ra\curly{Q}\ra 0
\end{equation}
with $\curly{Q}$ a skyscraper sheaf. Hence, it suffices to consider case $D=1/p^i[t]$. Let $a\in\ds{G}_{m,K}$ be the corresponding element to $t\in\curly{T}$, and $e$ correspond to $1$. Since, $\times a$ induces isomorphism on $\mathcal{T}$, there is an isomorphism between complexes $\curly{L}(1/p^i[e])$ and $\curly{L}(1/p^i[t])$, hence consider the case $t=e$, that is we can now work with factor $(X-1)$. Now, consider the \v{C}ech complex \eqref{surjComp} where it is shown that $d$ is surjective giving $H^j(\curly{T},\curly{L}(1/p^i[e]))=0$ for $j\geq 1$.

\end{proof}

\subsection{Perfectoid Elliptic Curves}
The above can be used to construct space of the form $\curly{L}(n[e])$ and come up with an elliptic curve(s) (with $\lambda_1\neq 0$) as given in \cite[pp 126]{fresnel2012rigid} or in \cite[pp 59]{silverman2013arithmetic}. 
\begin{equation}\label{ellipticCurve}
\begin{aligned}
Y^2+\lambda_1X^3+\lambda_2XY+\lambda_3X^2+\lambda_4Y+\lambda_5X+\lambda_6&=0\text{ with }i=0\\
Y^{2/p}+\lambda_1X^{3/p}+\lambda_2X^{1/p}Y^{1/p}+\lambda_3X^{2/p}+\lambda_4Y^{1/p}+\lambda_5X^{1/p}+\lambda_6&=0\text{ with }i=1\\
Y^{2/p^2}+\lambda_1X^{3/p^2}+\lambda_2X^{1/p^2}Y^{1/p^2}+\lambda_3X^{2/p^2}+\lambda_4Y^{1/p^2}+\lambda_5X^{1/p^2}+\lambda_6&=0\text{ with }i=2\\
\vdots\hspace{50mm}&=\vdots\qquad\qquad\vdots
\end{aligned}
\end{equation}

 \section{Theta Function} The basic Theta function is described in \eqref{funtheta} or \cite[pp 128]{fresnel2012rigid} is adapted to the perfectoid case by setting $n\in\ds{Z}[1/p]$
 \begin{equation}
 \Theta(X)=\prod_{n\geq 0}\left(1-\frac{q^n}{X}\right)\prod_{n>0}(1-q^n X)
 \end{equation}
The corresponding divisor is $\sum_{n\in\ds{Z}[1/p]}[q^n]$ with functional equation $\Theta(X/q)=-X\Theta(X)$. In particular for a divisor $D=\sum_in_i[x_i]$ with $n_i\in\ds{Z}[1/p]$, one defines $\Theta_D=\prod_i\Theta^{n_i}_{x_i}$ with divisor as $\sum_{i,n\in\ds{Z}[1/p]}n_i[q^nx_i]$. This can be used to prove that the following sequence (corresponding to Proposition 5.1.7 \cite[pp 128]{fresnel2012rigid}) is exact for $K$ perfectoid
\begin{equation}
0\ra \tio{K}\ra\curly{M}(\mathcal{T})\xra{\text{div}}\mathrm{Div}(\mathcal{T})\ra\ds{Z}[1/p]\times \mathcal{T}\ra 0
\end{equation}
 
\section{Weierstra{\ss} Equations}

The Weierstra{\ss} series is given as
\begin{equation}
\begin{aligned}
\wp(X)&=\sum_{n\in\ds{Z}[1/p]}\frac{q^nX}{(1-q^nX)^2}\\
\frac{d}{dX}\frac{q^nX}{(1-q^nX)^2}&=\frac{q^n+q^{2n}X}{(1-q^nX)^3}\\
\wp'(X)&=\frac{1}{2}\left( X\frac{d}{dX}\wp(X)-\wp(X)\right)=\sum_{n\in\ds{Z}[1/p]}\frac{q^{2n}X^2}{(1-q^nX)^3}
\end{aligned}
\end{equation}

The bijective map $\ds{Z}[1/p]\ra\ds{Z}[1/p]$ given by $n\mapsto n-1$ shows that $\wp(q^{-1}X)=\wp(X)$ and $\wp'(q^{-1}X)=\wp'(X)$. The only problem is that the above series might not even converge for $n\in\ds{Z}[1/p]$, although it converges for $n\in\ds{Z}$. Hence, the equation of the form \eqref{Roq24} does not hold much meaning.

\begin{equation}\label{Roq24}
\wp'^2+\wp\wp'=\wp^3+A\wp^2+B\wp+C\qquad A,B,C\in K
\end{equation}

But, we can still interpret the above as the sum of all elliptic curves as given in \eqref{ellipticCurve}, where the correspondence is given by choosing $q^{1/p^i}$. For $i=0$ in $\wp$ we get the standard curve, for $i=1$ or $q^{1/p}$ we get the curve with powers $1/p$. For $i=2$ we get the curve for $i=2$ or $q^{1/p^2}$ and so on.

Further work can be carried out by extending results in \cite[Chapter V, \S3-4]{silverman2013advanced} to the perfectoid case by setting $q\in\ds{Z}[1/p].$

\bibliographystyle{apalike}

\bibliography{\myreferences}

\end{document}